\documentclass[12pt]{article}

\usepackage{color}

\usepackage{amsmath,amsthm,amssymb}
\allowdisplaybreaks
\usepackage{mathrsfs}
\usepackage{cite}

\usepackage[pdftex,bookmarksopen=normd@maths.usyd.edu.autrue,bookmarks=true,unicode,setpagesize]{hyperref}
\hypersetup{colorlinks=true,linkcolor=black,citecolor=black}

\usepackage{color}

\newtheorem{theorem}{Theorem}
\newtheorem{corollary}[theorem]{Corollary}

\theoremstyle{remark}
\newtheorem{definition}[theorem]{Definition}
\theoremstyle{remark}
\newtheorem{example}[theorem]{Example}
\theoremstyle{remark}
\newtheorem{remark}[theorem]{Remark}

\newcommand{\R}{{\mathbb R}}
\newcommand{\N}{{\mathbb N}}

\newcommand{\K}{{\mathbb K}}
\newcommand{\M}{{\mathbb M}}
\newcommand{\la}{\langle}
\newcommand{\ra}{\rangle}
\newcommand{\Rp}{{\mathbb R_+}}

\newcommand{\FCC}{\mathcal{FC}(\Gamma(\hat X))}

\newcommand{\FCK}{\mathcal{FC}(\K(X))}

\begin{document}

\begin{center}{\Large \bf
Laplace operators on the cone of Radon measures
}\end{center}

{\large Yuri Kondratiev}\\
 Fakult\"at f\"ur Mathematik, Universit\"at
Bielefeld, Postfach 10 01 31, D-33501 Bielefeld, Germany;  NPU, Kyiv, Ukraine\\
 e-mail:
\texttt{kondrat@mathematik.uni-bielefeld.de}\vspace{2mm}

{\large Eugene Lytvynov}\\ Department of Mathematics,
Swansea University, Singleton Park, Swansea SA2 8PP, U.K.\\
e-mail: \texttt{e.lytvynov@swansea.ac.uk}\vspace{2mm}

{\large Anatoly Vershik }\\
 Laboratory of Representation Theory
and Dynamical Systems,
St.Petersburg Department of Steklov Institute of Mathematics,
27 Fontanka, St.Petersburg 191023, Russia; Department of Mathematics, St.\ Petersburg State University, St. Peterhof, Bibliotechnaya 2, Petersburg, 198908, Russia; Institute for Problems of Transmission and Information, B. Karetnyi 19, Moscow, 127994, Russia\\
 e-mail:
\texttt{vershik@pdmi.ras.ru}\vspace{4mm}

{\small
\begin{center}
{\bf Abstract}
\end{center}

\noindent
We consider the infinite-dimensional Lie group $\mathfrak G$ which is the semidirect product of the group of compactly supported diffeomorphisms of a Riemannian manifold $X$ and the commutative multiplicative group of functions on $X$. The group $\mathfrak G$ naturally acts on the space $\M(X)$ of Radon measures on $X$. We would like  to define a Laplace operator associated with   a natural representation of $\mathfrak G$ in  $L^2(\M(X),\mu)$. Here $\mu$ is assumed to be  the law of a measure-valued L\'evy process.
A unitary representation of the group cannot be determined, since the measure $\mu$ is not quasi-invariant with respect to the action of the group $\mathfrak G$. Consequently,
 operators of a representation of the Lie algebra and its universal enveloping algebra (in particular, a Laplace operator)
 are not defined.
 Nevertheless, we determine the  Laplace operator by using a special property of the action of the group $\mathfrak G$ (a partial quasi-invariance).
 We further prove the essential self-adjointness of the Laplace operator. Finally, we explicitly construct a diffusion process on $\M(X)$ whose generator is the Laplace operator.}
\vspace{2mm}

 \vspace{2mm}


\newpage

\section{Introduction}\label{drre6eiu}

In the representation theory, so-called quasi-regular representations of a group $\mathfrak G$ in a space  $L^2(\Omega,\mu)$
play an important role. Here  $\Omega$ is a homogeneous space for the group $\mathfrak G$ and $\mu$ is a (probability) measure  that is quasi-invariant with respect to the action of $\mathfrak G$ on $\Omega$. However, in the case studied in this paper as well as in similar cases, the measure is not quasi-invariant  with respect to the action of the group, so that one cannot define a quasi-regular unitary representation of the group. Hence, the problem of construction of a representation of the Lie algebra, of a Laplace operator and other operators from the universal enveloping algebra is highly non-trivial. Moreover,
we will  deal with the situation in which the representation of the Lie algebra cannot be
realized but, nevertheless, the Laplace operator may be defined correctly.

An important example of  a quasi-regular representation is the following case. Let $X$ be a smooth, noncompact Riemannian manifold, and let $\mathfrak G=\operatorname{Diff}_0(X)$, the group of $C^\infty$ diffeomorphisms of $X$ that are equal to the identity outside a compact set.
 Let $\Omega$ be the space $\Gamma(X)$ of locally finite subsets (configurations) in $X$, and let $\mu$ be the Poisson measure on $\Gamma(X)$. Then the Poisson measure is quasi-invariant with respect to the action of $\operatorname{Diff}_0(X)$ on $\Gamma(X)$, and the  corresponding unitary representation of $\operatorname{Diff}_0(X)$ in the $L^2$-space of the Poisson measure  was studied in \cite{VGG}, see also \cite{GGPS}. Developing analysis associated with this representation, one naturally arrives at a differential structure on the configuration space $\Gamma(X)$, and derives a Laplace operator on $\Gamma(X)$, see \cite{AKR}.
 In fact, one gets a certain lifting of the differential structure of the manifold $X$ to the configuration space $\Gamma(X)$. Hereby the Riemannian volume on $X$ is lifted to the Poisson measure on $\Gamma(X)$, and the Laplace--Beltrami operator on $X$, generated by the Dirichlet integral with respect to the Riemannian volume, is lifted to the generator of the Dirichlet form of the Poisson measure. The associated diffusion can be described as a Markov process on $\Gamma(X)$  in which movement of each point of configuration is a  Brownian motion in $X$, independent of the other points of the configuration, see \cite{KLR,Surgailis2,Surgailis}.


Let $C_0(X\to\Rp)$ denote the multiplicative group of continuous functions on $X$ with values in $\Rp:=(0,\infty)$ that are equal to one outside a compact set. (Analogously, we could have considered $C_0(X)$, the additive group of real-valued continuous functions on $X$ with compact support.)
The group of diffeomorphisms, $\operatorname{Diff}_0(X)$, naturally acts on $X$, hence on $C_0(X\to\Rp)$. In this paper, we will consider the group
$$\mathfrak G=\operatorname{Diff}_0(X) \rightthreetimes C_0(X\to\Rp),$$ the semidirect product of $\operatorname{Diff}_0(X)$ and
$C_0(X\to\Rp)$. This group and similar semidirect products and their representations play a fundamental role in mathematical physics and quantum field theory. Even more important than $\mathfrak G$ are the semidirect products in which the space $C_0(X\to\Rp)$ is replaced by a current space, i.e., a space of functions on $X$ with values in a Lie group. Note that, in our case, this Lie group, $\Rp$, is commutative.
The case of a commutative group, studied in this paper, is also important and it is interesting to find  relations with the theory of random fields and with infinite dimensional dynamical systems.

A wide class of representations of a group like $\mathfrak G$ is obtained by considering a probability measure on a space of
locally finite configurations
 These studies were initiated in
\cite{VGG}, and almost at the same time in \cite{GGPS}, but in less generality. See also \cite{VGG1,VGG3,VGG2,KLU}.

  The group $\mathfrak G$ naturally acts on the space $\M(X)$ of Radon measures on $X$. So a natural question is to identify a class of laws of random measures (equivalently, probability measures on $\M(X)$) which are quasi-invariant with respect to the action of the group $\mathfrak G$ and which allow for corresponding analysis, like Laplace operator,  diffusion in $\M(X)$, etc.   We will search for such random measures within the class of laws of measure-valued L\'evy processes whose intensity measure is infinite.
 Each  measure $\mu$ from this class is concentrated on the set $\K(X)$ of discrete Radon measures of the form $\sum_{i=1}^\infty s_i\delta_{x_i}$,
 where $\delta_{x_i}$ is the Dirac measure with mass at $x_i$ and $s_i>0$. Furthermore, $\mu$-almost surely, the configuration $\{x_i\}_{i=1}^\infty$ is dense in $X$, in particular, the set $\{x_i\}_{i=1}^\infty$  is not locally finite.

 A noteworthy example of a measure from this class is the gamma measure.
 In the case where $X$ is compact, it was proven  in  \cite{TsVY} that the gamma measure  is the unique law of a measure-valued L\'evy process which is quasi-invariant with respect to the action of the group $C_0(X\to\Rp)$ and which admits an equivalent $\sigma$-finite measure which is projective invariant (i.e., invariant up to a constant factor) with respect to the action of $C_0(X\to\Rp)$.
The latter ($\sigma$-finite) measure was studied in  \cite{Vershik} and was called there  the infinite dimensional Lebesgue measure.
 See also the references in \cite{Vershik} and\cite{Vershik2,Vershik3}.
 We also note that, in papers \cite{TsVY,VGG1,VGG2,VGG3}, the gamma measure was used in the representation theory of the group $SL(2,F)$, where $F$ is an algebra of functions on a manifold.

 In this paper, we first single out a class of laws of  measure-valued L\'evy processes $\mu$ which are quasi-invariant with respect to the action of the  group $C_0(X\to\Rp)$, compare with \cite{LS,vRYZ}. However, since the intensity measure of $\mu$ is infinite,
 the measure $\mu$ is not quasi-invariant with respect to the action of the diffeomorphsim group  $\operatorname{Diff}_0(X)$, and, consequently, it is not quasi-invariant with respect to the action of the group $\mathfrak G$.
 Thus, we do not have a quasi-regular representation of $\mathfrak G$ in $L^2(\K(X),\mu)$.

 Nevertheless, the action of the group $\mathfrak G$ on $\K(X)$ allows us to introduce the notion of a directional derivative on $\K(X)$, a tangent space, and a gradient. Furthermore, we introduce the notion of  partial quasi-invariance of a measure with respect to the action of a group. We show that the measure  $\mu$ is partially quasi-invariant with respect to $\mathfrak G$, and this essentially allows us to construct an associated Laplace operator.

 We note that, for each measure $\mu$ under consideration, we obtain a quasi-regular representation of the group $C_0(X\to\Rp)$ on $L^2(\K(X),\mu)$ and a corresponding integration
 by parts formula. Furthermore, there exists a filtration $(\mathcal F_n)_{n=1}^\infty$ on $\K(X)$ such that the $\sigma$-algebras $\mathcal F_n$ generate the $\sigma$-algebra on which
 the measure $\mu$ is defined, and hence the union of the spaces $L^2(\K(X),\mathcal F_n,\mu)$ is dense in $L^2(\K(X),\mu)$. The action of the group $\operatorname{Diff}_0(X)$ on $\K(X)$ transforms each $\sigma$-algebra $\mathcal F_n$ into itself, and the restriction of $\mu$ to $\mathcal F_n$ is quasi-regular with respect to the action of $\operatorname{Diff}_0(X)$. This implies a quasi-regular representation of  $\operatorname{Diff}_0(X)$ in $L^2(\K(X),\mathcal F_n,\mu)$, and we also obtain an integration by parts formula on this space. It should be stressed that the $\sigma$-algebras $\mathcal F_n$ are not invariant with respect to the action of the group $C_0(X\to\Rp)$. Despite the absence of a proper integration by parts formula related to the Lie algebra $\mathfrak g$ of the Lie group $\mathfrak G$, using the above facts, we arrive at a proper  Laplace operator related to the Lie algebra $\mathfrak g$, and this Laplace operator is self-adjoint in $L^2(\K(X),\mu)$.

 We next prove that the Laplace operator is essentially self-adjoint on a set of test functions.
  Assuming that the dimension of the manifold $X$ is $\ge 2$, we then  explicitly construct a diffusion process on $\K(X)$ whose generator is the Laplace operator.

 Finally, we notice that a different natural choice of a tangent space leads to a different, well defined Laplace operator in $L^2(\K(X),\mu)$. Using the theory of Dirichlet forms, we can prove that the corresponding diffusion process in $\K(X)$ exists. However, its explicit construction is still an open problem, even at a heuristic level.

 \section{Partial quasi-invariance}

 \subsection{The group $\mathfrak S$}
 Let $X$ be a separable, connected, oriented $C^\infty$ (non-compact) Riemannian manifold. Recall that $\operatorname{Diff}_0(X)$ denotes the group of diffeomorphisms of $X$ which are equal to the identity outside a compact set, and $C_0(X\to\Rp)$ denotes the multiplicative group of continuous functions on $X$ with values in $\Rp$ which are equal to one outside a compact set.
 The group $\operatorname{Diff}_0(X)$ acts on $C_0(X\to\Rp)$ by automorphisms:
for each $\psi\in \operatorname{Diff}_0(X)$,
$$C_0(X\to\Rp)\ni \theta \mapsto \alpha(\psi)\theta:=\theta\circ\psi^{-1}\in C_0(X\to\Rp).$$
We denote
$$\mathfrak G:= \operatorname{Diff}_0(X) \rightthreetimes C_0(X\to\Rp),$$
the semidirect product of $\operatorname{Diff}_0(X)$ and $C_0(X\to\Rp)$ with respect to $\alpha$. Thus, as a set, $\mathfrak G$ is equal to the Cartesian product of   $\operatorname{Diff}_0(X)$ and $C_0(X\to\Rp)$, and the multiplication in $\mathfrak G$ is given by
$$g_1g_2=(\psi_1\circ\psi_2,\, \theta_1(\theta_2\circ\psi_1^{-1}))\quad\text{for $g_1=(\psi_1,\theta_1),\, g_2=(\psi_2,\theta_2)\in \mathfrak G$.}$$

Let $\mathcal B(X)$ denote the Borel $\sigma$-algebra on $ X$.
 Let $\M(X)$ denote the space of all  Radon measures on $(X,\mathcal B(X))$. The space $\mathbb M(X)$ is equipped with the vague topology.

 The group $\mathfrak G$  naturally acts on $\mathbb M(X)$:  for any $g=(\psi,\theta)\in\mathfrak G$ and any $\eta\in\mathbb M(X)$, we define the Radon measure $g\eta$ by
\begin{equation}\label{uy87r8o7f}
 d(g\eta)(x):=\theta(x)\,d(\psi^*\eta)(x).\end{equation}
Here $\psi^*\eta$ is the pushforward of $\eta$ under $\psi$. In particular, each $\psi\in \operatorname{Diff}_0(X)$ acts on $\M(X)$ as $\eta\mapsto\psi^*\eta$, while each $\theta\in C_0(X\to\Rp)$ acts on $\M(X)$ as $\eta\mapsto\theta\cdot \eta$.

We recall that a probability measure $\mu$ on $(\M(X),\mathcal B(\M(X)))$ is called (the law  of) a random measure, see e.g.\ \cite[Chap.~1]{Kal}. If a random measure $\mu$ is quasi-invariant with respect to the action of $\mathfrak G$ on $\M(X)$, then we get a so-called quasi-regular (unitary) representation of $\mathfrak G$ in $L^2(\M(X),\mu)$, given by
$$ (U_gF)(\eta)=F(g^{-1}\eta)\sqrt{\frac{d\mu^{g    }}{d\mu}(\eta)},\quad g\in\mathfrak G.$$
Here $\mu^g$ is the pushforward of $\mu$ under $g$. Clearly, the quasi-invariance of $\mu$ with respect to the action of $\mathfrak G$ is equivalent to the quasi-invariance of $\mu$ with respect to the action of both groups  $\operatorname{Diff}_0(X)$ and $C_0(X\to\Rp)$.

We will search for quasi-invariant measures $\mu$ within the following class of measures.

\subsection{Measure-valued L\'evy processes} \label{syigufuq}

\begin{definition} \label{fdrtder6de6} Let $\lambda$ be a measure on $\Rp$ which satisfies
\begin{equation}\label{uyr76r586trgfrr}\int_{\Rp}\min\{1,s\}\,d\lambda(s)<\infty.\end{equation}
The {\it law of a   measure-valued L\'evy process on $X$ with intensity measure $\lambda$} is defined as the unique probability measure $\mu_\lambda$ on $(\M(X),\mathcal B(\M(X)))$ which has the Fourier transform
\begin{equation}\label{tyeu64e37ord}\int_{\mathbb M(X)}e^{i\la \varphi,\eta\ra}\, d\mu_\lambda(\eta)=\exp\bigg[
\int_X
\int_{\Rp}(e^{is\varphi(x)}-1)\,d\lambda(s)
\,dx \bigg],\quad \varphi\in C_0(X).\end{equation}
Here, $C_0(X)$ denotes the space of continuous functions on $X$ with compact support,  $\la \varphi,\eta\ra:=\int_X \varphi\,d\eta$, and $dx$ denotes the Riemannian volume on $X$.
\end{definition}

The existence of the measure $\mu_\lambda$ from Definition~\ref{fdrtder6de6} follows from Kingman \cite{Kingman1}.
Note that the measure $\mu_\lambda$ has the property that, for any mutually disjoint  sets
$A_1,\dots,A_n\in\mathcal B_0(X)$,  the random variables
$\eta(A_1),\dots,\eta(A_n)$ are independent. Furthermore, for any  $A_1, A_2\in\mathcal B_0(X)$ in $X$ such that $\int_{A_1}dx=\int_{A_2}dx$, the random variables $\eta(A_1)$ and $\eta(A_2)$ have the same distribution. Here and below, $\mathcal B_0(X)$ denotes the collection of all sets $A\in\mathcal B(X)$ whose closure is compact.

\begin{remark}
Note  that  $\mu_\lambda$ belongs to the class of probability measures on $\mathcal D'(X)$---the dual of the nuclear space $\mathcal D(X)=C_0^\infty(X)$---which
are called {\it generalized stochastic processes with independent values at  every point}. (Evidently, $\M(X)\subset\mathcal D'(X)$.) These probability measures were studied by Gel'fand and Vilenkin in \cite{GV}.
\end{remark}

Below, we will heavily use the following explicit construction of the measure $\mu_\lambda$.  We define a metric on $\Rp$ by
$$ d_{\Rp}(s_1,s_2):=\left|\log(s_1)-\log(s_2)\right|,\quad s_1,s_2\in\Rp.$$
Then $\Rp$ becomes a locally compact Polish space, and any set of the form $[a,b]$, with $0<a<b<\infty$, is compact.
We denote $\hat X:=\Rp\times X$ and define the  configuration space over $\hat X$ by
$$\Gamma(\hat X):=\big\{\gamma\subset \hat X\mid |\gamma\cap\Lambda|<\infty\text{ for each compact }\Lambda\subset \hat X\,\big\}.
$$
Here $|\gamma\cap\Lambda|$ denotes the cardinality of the  set $\gamma\cap\Lambda$. The space $\Gamma(\hat X)$ is endowed with the vague topology (after identification of $\gamma\in\Gamma(\hat X)$ with the Radon measure $\sum_{(s,x)\in\gamma}\delta_{(s,x)}$ on $\hat X$). We denote by $\pi_\varkappa$ the Poisson measure on $(\Gamma(\hat X),\mathcal B(\Gamma(\hat X)))$ with intensity measure
\begin{equation}\label{5}
d\varkappa(s,x):=d\lambda(s)\,dx,
\end{equation}
see e.g.\ \cite[Sec.~1.3]{Kal}.

We denote by $\Gamma_{pf}(\hat X)$ the measurable subset of $\Gamma(\hat X)$ which consists of all configurations $\gamma$ which satisfy:

\begin{itemize}
\item[(i)] ({\it pinpointing}) if $(s_1,x_1),(s_2,x_2)\in\gamma$ and $(s_1,x_1)\ne(s_2,x_2)$, then $x_1\ne x_2$;

\item[(ii)]  ({\it finite local mass}) for each  $A\in\mathcal B_0(X)$, $\displaystyle\sum_{(s,x)\in\gamma\cap (\Rp\times A)}s<\infty$.
\end{itemize}
Then, by the properties of the Poisson measure, $\pi_\varkappa(\Gamma_{pf}(\hat X))=1$.
We  construct a measurable mapping
$\mathscr R  :\Gamma_{pf}(\hat X)\to\mathbb M(X) $ by setting
\begin{equation}\label{uyfgutfrrsesa}
\Gamma_{pf}(\hat X)\ni \gamma=\{(s_i,x_i)\}\mapsto \mathscr R\gamma:=\sum_i s_i\delta_{x_i}\in \mathbb M(X),\end{equation}
see \cite[Theorem~6.2]{HKPR}.
Then the measure $ \mu_\lambda$ is the pushforward of the Poisson measure $\pi_\varkappa$ under $\mathscr R$.

We denote by $\K(X)$ the  cone of  discrete Radon measures on $X$:
$$\mathbb K(X):=\left\{
\eta=\sum_i s_i\delta_{x_i}\in \mathbb M(X) \mid s_i>0,\, x_i\in X
\right\}.$$
Here, the atoms $x_i$ are assumed to be distinct and their total number is at most countable.  We denote $\tau(\eta):=\{x_i\}$, i.e., the set on which the measure $\eta$ is concentrated. For $\eta\in\K(X)$ and $x\in\tau(\eta)$, we  denote by $s_x$ the mass of $\eta$ at  point $x$, i.e., $s_x
:=\eta(\{x\})$. Thus, each $\eta\in\K(X)$ can be written in the form $\eta=\sum_{x\in\tau(\eta)}s_x\delta_x$.
As shown in \cite{HKPR}, $\mathbb K(X)\in\mathcal B(\mathbb M(X))$. We denote by $\mathcal B(\mathbb K(X))$ the trace $\sigma$-algebra of $\mathcal B(\mathbb M(X))$ on $\K(X)$.
It follows from \eqref{uyfgutfrrsesa} that $\mathscr R$ is a bijective mapping between $\Gamma_{pf}(\hat X)$ and $\K(X)$, thus $\mu_\lambda(\K(X))=1$, i.e., we can consider $\mu_\lambda$ as a probability measure on $(\K(X),\mathcal B(\mathbb K(X)))$.

In the case where $\lambda(\Rp)<\infty$, $\mu_\lambda$ is, in fact, a marked Poisson measure. Indeed, in this case, $\mu_\lambda$ is concentrated on a subset of $\K(X)$ which consists of all $\eta\in\K(X)$ such that $\tau(\eta)$ is a locally finite configuration in $X$, i.e., $|\tau(\eta)\cap A|<\infty$ for each compact $A\subset X$. The Laplace operator related to the marked Poisson measures was studied in \cite{KLU}. So, in this paper, we will be interested in the (much more complicated) case where
\begin{equation}\label{guf7r}
\lambda(\Rp)=\infty.\end{equation}
 In this case, it can be shown that, with $\mu_\lambda$-probability one, $\tau(\eta)$ is a dense subset of $X$

Furthermore, we will assume that the measure $\lambda$ is absolutely continuous with respect to the Lebesgue measure, and let
\begin{equation}\label{vytd7}d\lambda(s)=\frac{l(s)}{s}\,ds,
\end{equation}
where $l:\Rp\to[0,\infty)$.
Under this assumption, condition \eqref{uyr76r586trgfrr}
becomes
\begin{equation}\label{tyd6e6i}
\int_{\Rp} l(s)\min\{1,s^{-1}\}\,ds<\infty.\end{equation}

\begin{example}\label{example3}
By choosing $l(s)=e^{-s}$, one obtains the gamma measure $\mu_\lambda$.  The Laplace transform of this measure is given by
$$\int_{\K(X)}e^{\la \varphi,\eta \ra}\,d\mu_\lambda(\eta)=\exp\left[-\int_{X}\log(1-\varphi(x))\,dx\right],\quad \varphi\in C_0(X),\ \varphi<1.$$
\end{example}

\subsection{Quasi-invariance with respect to $C_0^\infty(X\to\Rp)$}

We will now single out those measures $\mu_\lambda$ that are quasi-invariant with respect to the action of $C_0^\infty(X\to\Rp)$.

\begin{theorem}\label{guyr7r}
Assume that \eqref{vytd7} and \eqref{tyd6e6i} hold. Further assume  that
\begin{equation}\label{uyfr7r78r}l(s)>0\quad\text{for all }
s\in\Rp,
\end{equation}
 and  for each $n\in\N$, there exists $\varepsilon>0$ such that
\begin{equation}\label{rtew64}\sup_{r\in[1/n,\,n]}\int_{(0,\varepsilon)}|l(rs)-l(s)|\,s^{-1}\,ds<\infty.\end{equation}
Then the measure $\mu_\lambda$  is quasi-invariant with respect to all transformations from\linebreak  $C_0^\infty(X\to\Rp)$. More precisely, each $\theta\in C_0(X\to\Rp)$ maps $\mathbb K(X)$ into itself, and the pushforward of $\mu_\lambda$ under $\theta$, denoted by $\mu_\lambda^{\theta}$, is equivalent  to $\mu_\lambda$. Furthermore, the corresponding density is given by
\begin{align}
\frac{d\mu_\lambda^\theta}{d\mu_\lambda}(\eta)=\exp\bigg[\int_X&
\log\left(\frac{l(\theta^{-1}(x)s_x)}{l(s_x)}\right)s_x^{-1}\,d\eta(x)\notag\\
&+\int_X\int_{\Rp}\big(l(s)-l(\theta^{-1}(x)s)\big)s^{-1}\,ds\,dx\bigg].
\label{rtew6u4w6}\end{align}
\end{theorem}

\begin{proof}
We first note that, by \eqref{tyd6e6i} and  \eqref{rtew64},
\begin{align}
&\sup_{r\in[1/n,\,n]}\int_{\Rp}|l(rs)-l(s)|\,s^{-1}\,ds\notag\\
&\quad\le \sup_{r\in[1/n,\,n]}\bigg(\int_{(0,\varepsilon)}|l(rs)-l(s)|\,s^{-1}\,ds+\int_{[\varepsilon,\infty)}
l(rs)\,s^{-1}\,ds+\int_{[\varepsilon,\infty)}l(s)\,s^{-1}\,ds\bigg)\notag\\
&\quad\le \sup_{r\in[1/n,\,n]}\bigg(\int_{(0,\varepsilon)}|l(rs)-l(s)|\,s^{-1}\,ds\bigg)+2\int_{[\varepsilon/n,\,
\infty)}l(s)s^{-1}\,ds<\infty.\label{e645we6we6}
\end{align}

Recall the bijective mapping $\mathscr R:\Gamma_{pf}(\hat X)\to\K(X)$ and the Poisson measure $\pi_\varkappa$ on $\Gamma_{pf}(\hat X)$. Fix $\eta=\sum_i s_i\delta_{x_i}\in\mathbb K(X)$. Then
$\mathscr R^{-1}(\theta\eta)=\{(\theta(x_i)s_i,x_i)\}$.
Let $\pi_\varkappa^\theta$ denote the pushforward of the measure $\pi_\varkappa$ under the transformation
$\{(s_i,x_i)\}\mapsto \{(\theta(x_i)s_i,x_i)\}$.
Calculating the Fourier transform of the measure $\pi_\varkappa^\theta$, we easily see that
 $\pi_\varkappa^\theta$ is the Poisson measure over $\hat X$ with intensity measure
$$d\varkappa^\theta(s,x):=\frac{l(\theta^{-1}(x)s)}s\,ds\,dx.$$
Note that the measures $\varkappa^\theta$ and $\varkappa$ are equivalent, and
\begin{equation}\label{rtyed6ie7}
\frac{d\varkappa^\theta}{d\varkappa}(s,x)=\frac{l(\theta^{-1}(x)s)}{l(s)}>0.\end{equation}
By \eqref{e645we6we6}, we have, for the total variation of the signed measure $\varkappa^\theta-\varkappa$,
\begin{equation}\label{rtst56sw}
\int_{\hat X}\bigg|\frac{d\varkappa^\theta}{d\varkappa}-1\bigg|\,d\varkappa=\int_X\int_{\Rp}|l(\theta^{-1}(x)s)-l(s)|\,\frac1s\,ds\,dx<\infty.\end{equation}
Hence, we can apply Skorohod's result \cite{Sko} on the equivalence of Poisson measures, see also  \cite[Lemma~1]{Tak} where this result is extended to a rather general underlying space. Thus, by \eqref{rtyed6ie7} and \eqref{rtst56sw}, the Poisson measures $\pi_\varkappa^\theta$ and $\pi_\varkappa$ are equivalent and
\begin{equation}\label{tyrde6u4e76d}
\frac{d\pi_\varkappa^\theta}{d\pi_\varkappa}(\gamma)=\exp\left[
\left\langle\log\left(\frac{d\varkappa^\theta}{d\varkappa}\right),\gamma\right\rangle+\int_{\hat X}\left(1-\frac{d\varkappa^\theta}{d\varkappa}\right)d\varkappa
\right],\end{equation}
with $\log(\frac{d\varkappa^\theta}{d\varkappa})\in L^1(\hat X,\gamma)$ for $\pi_\varkappa$-a.a.\ $\gamma\in\Gamma(\hat X)$.
Noting that the measure $\mu_\lambda^\theta$ is the pushforward of $\pi_\varkappa^\theta$ under $\mathscr R$, we immediately get the statement of the theorem.
\end{proof}

Thus, for each measure $\mu_\lambda$ as in Theorem~\ref{guyr7r}, we get a quasi-regular representation of $C_0^\infty(X\to\Rp)$ in $L^2(\K(X),\mu_\lambda)$.

\begin{corollary}
\label{hufurti8tr8686} Assume that \eqref{vytd7}--\eqref{uyfr7r78r} hold. Further assume that, for some $\rho>0$, $l(s)=l_1(s)+l_2(s)$ for $s\in(0,\rho)$. Here, the function $l_1$ is differentiable on $(0,\rho)$
and for each $n\in\N$
\begin{equation}\label{se5w5w}\int
_{(0,\,\rho/n)}\sup_{u\in[s/n,\,sn]}|l_1'(u)|\,ds<\infty
,\end{equation}
 while
 $l_2\in L^1((0,\rho),s^{-1}\,ds)$. Then condition \eqref{rtew64} is satisfied, and so the conclusion of Theorem~\ref{guyr7r} holds.
\end{corollary}

\begin{remark}
Note that condition \eqref{se5w5w} is  stronger than $\int_{(0,\rho)}|l_1'(s)|\,ds<\infty$.
\end{remark}

\begin{proof}[Proof of Corollary~\ref{hufurti8tr8686}] We only need to check that condition \eqref{rtew64} is satisfied. But for the function $l_1$, the fulfillment of such a condition easily follows from Taylor's formula and \eqref{se5w5w}, while for the function $l_2$, the proof is similar to the arguments used in  \eqref{e645we6we6}.
\end{proof}

\begin{example} In the case of the gamma measure, the function $l(s)=l_1(s)=e^{-s}$ clearly satisfies
the conditions of Corollary~\ref{hufurti8tr8686}.
Formula~\eqref{rtew6u4w6} now becomes
$$
\frac{d\mu_\lambda^\theta}{d\mu_\lambda}(\eta)=\exp\bigg[\int_X(1-\theta^{-1}(x))\,d\eta(x)-\int_X \log(\theta(x))\,dx\bigg],
$$
compare with \cite[Theorem~3.1]{TsVY}.
\end{example}

\begin{example}\label{tyre56i54ei}
By analogy with  \cite[Comment 2 to Theorem~1]{LS},
let us  consider a function $l(s)$  such that, for some $\rho\in(0,1)$, and $\alpha>0$
\begin{equation}\label{utyfr65e6afxt}
l(s)=(-\log s)^{-\alpha},\quad s\in(0,\rho),\end{equation}
\eqref{uyfr7r78r} holds and  $\int_{[\rho,\infty)}l(s)s^{-1}\,ds<\infty$. Since  $l(s)$ is bounded on $(0,\rho)$, we get $\int_{(0,\rho)}l(s)\,ds<\infty$.
 For each $n\in\N$,
 \begin{align*}
\int_{(0,\,\rho/n)}\,\sup_{u\in[s/n,\,sn]}|l'(u)|\,ds
&= \int_{(0,\,\rho/n)}\,\sup_{u\in[s/n,\,sn]}\big(\alpha(-\log u)^{-\alpha-1}\,u^{-1}\big) \,ds\\
&\le \alpha n \int_{(0,\,\rho/n)}(-\log(sn))^{-\alpha-1}\,s^{-1}\,ds<\infty.
\end{align*}
 Hence, $l$ satisfies the assumptions of Corollary~\ref{hufurti8tr8686}. Note also that, for $\alpha\in(0,1]$, we get $\lambda(\Rp)=\infty$.
\end{example}

\subsection{Partial quasi-invariance with respect to $\mathfrak G$}

Analogously to the proof of Theorem~\ref{guyr7r}, we conclude that  a measure $\mu_\lambda$ is quasi-invariant with respect to the action of $\operatorname{Diff}_0(X)$ if and only if, for each $\psi\in \operatorname{Diff}_0(X)$, we have $\sqrt{J_\psi(x)}-1\in L^2(\hat X, d\lambda(s)\,dx)$, where $J_\psi(x)$ is the Jacobian determinant of $\psi$ (with respect to the Riemannian volume $dx$). Obviously, this condition is satisfied if and only if $\lambda(\Rp)<\infty$. So, in the case where \eqref{guf7r} holds, $\mu_\lambda$ is not quasi-invariant with respect to  $\operatorname{Diff}_0(X)$ and $\mathfrak G$. Because of this, we will now introduce a weaker notion of partial quasi-invariance.

\begin{definition}\label{uit886} Let $(\Omega,\mathscr F,P)$ be a probability space, and let $\mathscr G$ be a group which acts on $\Omega$. We will say that the probability measure $P$ is {\it  partially quasi-invariant with respect to transformations $g\in \mathscr G$} if there exists a filtration $(\mathscr F_n)_{n=1}^\infty$ such that:

\begin{itemize}
\item[(a)] $\mathscr F$ is the minimal $\sigma$-algebra on $\Omega$ which contains all $\mathscr F_n$, $n\in\N$;

\item[(b)] For each $g\in\mathscr G$ and $n\in\mathbb N$, there exists $m\in\mathbb N$ such that $g$ maps $\mathscr F_n$ into $\mathscr F_m$\,;

\item[(c)] For any $n\in\N$ and  $g\in \mathscr G$,
there exists a measurable function $R_g^{(n)}:\Omega\to[0,\infty]$ such that, for each $F:\Omega\to [0,\infty]$ which is $\mathscr F_n$-measurable,
$$\int_\Omega F(g\omega)\,dP(\omega)=\int_\Omega F(\omega)\,R_g^{(n)}(\omega)\,dP(\omega).$$
\end{itemize}
\end{definition}

\begin{remark}Note that, if $\mathscr F_n$ is a proper subset  of $\mathscr F$, the function $R_g^{(n)}$ is not uniquely defined. It becomes unique ($P$-almost surely) if we additionally require $R_g^{(n)}$ to be $\mathscr F_n$-measurable.
\end{remark}

\begin{remark}Clearly, if a probability measure $P$ is quasi-invariant with respect to $\mathscr G$, then it is partially quasi-invariant. Indeed, we may choose $\mathscr F_n=\mathscr F$ for all $n\in\mathbb N$, and set $R_g=R_g^{(n)}=\frac{dP^g}{dP}$, where $P^g$ is the pushforward of $P$ under $g\in G$.\end{remark}

\begin{remark} In the case where a probability measure $P$ is only partially quasi-invariant with respect to a Lie group $\mathscr G$, we have no true unitary representation of  $\mathscr G$, and, consequently, no representation of the Lie algebra and its universal enveloping algebra. However, we may have an integration by parts formula for the measure $P$ in a weak form, see  Section~\ref{tyd6rte} below.
\end{remark}

\begin{theorem}\label{yur767}
Let \eqref{guf7r} hold and let the conditions of Theorem~\ref{guyr7r} be satisfied. Then, the measure $\mu_\lambda$ is partially quasi-invariant with respect to the action of the group $\mathfrak G$.
\end{theorem}

\begin{proof} The Borel $\sigma$-algebra $\mathcal B(\Gamma_{pf}(\hat X))$ may be identified as the minimal $\sigma$-algebra on $\Gamma_{pf}(\hat X)$ with respect to which each mapping of the following form is measurable:
\begin{equation}\label{huguyl8t8}
\Gamma_{pf}(\hat X)\ni\gamma\mapsto |\gamma\cap\Lambda|,\quad \Lambda\in\mathcal B_0(\hat X),
\end{equation}
see e.g.\ Section~1.1, in particular Lemma~1.4, in \cite{Kal}. For each $n\in\N$, we denote by $\mathcal B_n(\Gamma_{pf}(\hat X))$ the minimal $\sigma$-algebra on $\Gamma_{pf}(\hat X)$ with respect to which each mapping of the form \eqref{huguyl8t8} is measurable, with $\Lambda$ being a  subset of  $[1/n,\infty)\times X $. Clearly, $(\mathcal B_n(\Gamma_{pf}(\hat X)))_{n=1}^\infty$ is a filtration and  $\mathcal B(\Gamma_{pf}(\hat X))$ is the minimal $\sigma$-algebra on $\Gamma_{pf}(\hat X)$ which contains all $\mathcal B_n(\Gamma_{pf}(\hat X))$. Next, we  denote by $\mathcal B_n(\K(X))$ the image of $\mathcal B_n(\Gamma_{pf}(\hat X))$
under the mapping $\mathscr R$.
By \cite[Theorem~6.2]{HKPR},
$$
\mathcal B(\mathbb K(X))=\big\{\mathscr R A\mid A\in\mathcal B(\Gamma_{pf}(\hat X))\big\}.$$
Hence,  $(\mathcal B_n(\K(X)))_{n=1}^\infty$ is a filtration and $\mathcal B(\K(X))$ is the minimal $\sigma$-algebra on $\K(X)$ which contains all $\mathcal B_n(\K(X))$.

Let $n\in\N$ and $g=(\psi,\theta)\in\mathfrak G$. Choose $m\in\mathbb N$ such that $\frac1m\le\frac1n\,\inf_{x\in X}\theta(x)$. Then $g$ maps $\mathcal B_n(\K(X))$ into $\mathcal B_m(\K(X))$. Furthermore, let $F:\K(X)\to[0,\infty]$ be measurable with respect to  $\mathcal B_n(\K(X))$. By \eqref{uy87r8o7f},
\begin{equation}\label{buig8it8}
 \int_{\K(X)} F(g\eta)\,d\mu_\lambda(\eta)=\int_{\K(X)}F(\theta\cdot\eta)\,d\mu^\psi_\lambda(\eta) \end{equation}
where $\mu^\psi_\lambda$ is the pushforward of $\mu_\lambda$ under $\psi^*$. The function
$\eta\mapsto F(\theta\cdot\eta)$ is $\mathcal B_m(\K(X))$-measurable. As easily seen, $\psi^*$ maps $\mathcal B_m(\K(X))$ into itself, the restriction of the measure $\mu_\lambda^\psi$ to $\mathcal B_m(\K(X))$  is absolutely continuous with respect to the restriction of $\mu_\lambda$ to $\mathcal B_m(\K(X))$, and the corresponding density is given by $\prod_{x\in\tau(\eta):\, s_x\ge\frac1m}J_\psi(x)$.
Hence, by \eqref{buig8it8} and Theorem~\ref{guyr7r},
\begin{align*}  \int_{\K(X)} F(g\eta)\,d\mu_\lambda(\eta)&=\int_{\K(X)} F(\theta\cdot \eta)\,\prod_{x\in\tau(\eta):\, s_x\ge\frac1m}J_\psi(x)\,d\mu_\lambda(\eta)\\
&=\int_{\K(X)} F(\eta)\,\prod_{x\in\tau(\eta):\, s_x\ge\frac{\theta(x)}m}J_\psi(x)\,d\mu^\theta_\lambda(\eta)\\
&=\int_{\K(X)} F(\eta) R_g^{(n)}(\eta)\,d\mu_\lambda(\eta),
\end{align*}
where
\begin{equation}\label{hjvfut7r678}
 R_g^{(n)}(\eta)= \prod_{x\in\tau(\eta):\, s_x\ge\frac{\theta(x)}m}J_\psi(x)\cdot\frac{d\mu_\lambda^\theta}{d\mu_\lambda}(\eta),\end{equation}
with $\frac{d\mu_\lambda^\theta}{d\mu_\lambda}(\eta)$ being given by \eqref{rtew6u4w6}.
\end{proof}

\section{Integration by parts}\label{tyd6rte}
Let us first make a general observation about partial quasi-invariance.
 Assume that $\mathscr G$ is a Lie group which acts on $\Omega$ and assume that a probability measure $P$ on $\Omega$ is partially quasi-invariant with respect to $\mathscr G$. Let $\mathbf g$ be the Lie algebra of $\mathscr G$. Fix any $v\in\mathbf g$ and  let $(g^v_t)_{t\in\mathbb R}$ be the corresponding one-parameter subgroup of $\mathscr G$. For a function $F:\Omega\to\R$, we may now introduce a derivative of $F$ in direction $v$ by
 $\nabla^{\mathscr G}_v F(\omega):=\frac{d}{dt}\big|_{t=0}F(g^v_t\omega)$. Fix any $n\in\mathbb N$, and assume that there exists $m\ge n$ such that,  for all $t$ from a neighborhood of zero,
$g_t^v$ maps $\mathscr F_n$ to $\mathscr F_m$. Then, at least heuristically, we get, for  $\mathscr F_n$-measurable, differentiable functions $F,G:\Omega\to\R$:
\begin{align}
&\int_\Omega \nabla^{\mathscr G}_v F(\omega)G(\omega)\,dP(\omega)=\frac d{dt}\Big|_{t=0}\int_\Omega F(g^v_t\omega)G(g^v_tg^v_{-t}\,\omega)\,dP(\omega)\notag\\
&\quad=\frac d{dt}\Big|_{t=0}\int_\Omega F(\omega) G(g^v_{-t}\,\omega)R^{(m)}_{g_t^v}(\omega)\,dP(\omega)\notag\\
&\quad =-\int_\Omega F(\omega)\nabla^{\mathscr G}_v G(\omega)\,dP(\omega)-\int_\Omega F(\omega) G(\omega) B_v^{(m)}(\omega)\,dP(\omega),\label{huvfrt7r75}
\end{align}
where
\begin{equation}\label{uyfr7r7r}B_v^{(m)}(\omega):=-\frac{d}{dt}\Big|_{t=0} R^{(m)}_{g_t^v}(\omega).
\end{equation}
Note that the function $B_v^{(m)}$ in formula \eqref{huvfrt7r75} can be replaced by the conditional expectation of $B_v^{(m)}$ with respect to the $\sigma$-algebra $\mathscr F_m$, which is equal to $B_v^{(n)}$.  Thus, we get an integration by parts formula in a weak form. We will now rigorously derive such a formula in the case of the group $\mathfrak G$.

The Lie algebra of the Lie group $\operatorname{Diff}_0(X)$ is the space $\operatorname{Vect}_0(X)$ consisting of all $C^\infty$-vector fields (i.e., smooth sections of   $T(X)$\,) which have compact support. For $v\in \operatorname{Vect}_0(X)$, let $(\psi_t^v)_{t\in\R}$ be the corresponding  one-parameter subgroup of $\operatorname{Diff}_0(X)$,  see e.g.\ \cite[Chap.~IV, Sect.~6 and 7]{Boo} and \cite[subsec.~3.1]{AKR}. The corresponding derivative of a function $F:\M(X)\to\R$ in direction $v$ will be denoted by $\nabla^{\M}_vF(\eta)$.

As the Lie algebra of $C_0(X\to\Rp)$ we may take the space $C_0(X)$. For each $h\in C_0(X)$, the corresponding one-parameter subgroup of $C_0(X\to\Rp)$ is given by $(e^{th})_{t\in\R}$. The corresponding derivative of a function $F:\M(X)\to\R$ in direction $h$ will be denoted by $\nabla^{\M}_hF(\eta)$.

Next, $\mathfrak g:= \operatorname{Vect}_0(X)\times C_0(X) $ can be thought of as a Lie algebra that corresponds to the Lie group $\mathfrak G$. For an arbitrary $(v,h)\in\mathfrak g$, we may consider the curve $\{(\psi_t^v,e^{th}),\,t\in\R\}$ in $\mathfrak G$. (Note that this curve does not form a subgroup of $\mathfrak G$.) The corresponding derivative of a function $F:\M(X)\to\R$ in direction $(v,h)$ will be denoted by $\nabla^{\M}_{(v,h)}F(\eta)$. We clearly have:
\begin{equation}\label{ytdtr6e56i}
(\nabla_{(v,h)}^{\mathbb M}F)(\eta)=(\nabla_{v}^{\M}F)(\eta)+(\nabla_{h}^{\M}F)(\eta)\end{equation}
(at least, under reasonable assumptions on $F$). Note that, in the above definitions we may take a function $F:\K(X)\to\R$.

Let us now introduce a set of `test' functions on $\K(X)$ such that each function $F$ from this set is measurable with respect to $\mathcal B_n(\K(X))$ for some $n\in\N$.
Denote by $C_0^\infty(\hat X)$ the space of all infinitely differentiable functions on $\hat X$ which have compact support in $\hat X$.  We denote by $\FCC$ the set of all cylinder functions $G:\Gamma(\hat X)\to\R$ of the form
\begin{equation}\label{giyfci}
G(\gamma)=g(\la \varphi_1,\gamma\ra,\dots,\la \varphi_N,\gamma\ra),
\quad\gamma\in\Gamma(\hat X),
\end{equation}
where $g\in C_{b}^\infty(\R^N)$, $\varphi_1\,\dots,\varphi_N\in C_0^\infty(\hat X)$, and $N\in\N$.  Here $C_{b}^\infty(\R^N)$ is the set of all infinitely differentiable functions on $\R^N$ which, together with all their derivatives, are bounded.
Next, we define
$$\FCK
 :=\big\{F:\K(X)\to\R\mid F(\eta)=G(\mathscr R^{-1}\eta)\text{ for some }G\in\FCC\big\}.$$
For $\varphi\in C_0^\infty(\hat X)$ and $\eta\in\K(X)$, we denote
\begin{equation}\label{tdrytd6}
 \langle\!\langle \varphi,\eta\rangle\!\rangle:=\la\varphi,\mathscr R^{-1}\eta\ra=\sum_{x\in\tau(\eta)} \varphi(s_x,x)=\int_X \varphi(s_x,x)\,d\tilde \eta(x).\end{equation}
Here $d\tilde\eta(x):=\frac1{s_x}\,d\eta(x)$, i.e., $\tilde\eta=\sum_{x\in\tau(\eta)}\delta_x$. (Note that $\tilde\eta$ is not a Radon measure.)
Then, each function $F\in\FCK$ has the form
\begin{equation}\label{fty6ed6u45} F(\eta)=g\big(\langle\!\langle \varphi_1,\eta\rangle\!\rangle,\dots,\langle\!\langle \varphi_N,\eta\rangle\!\rangle\big),\quad \eta\in\K(X),\end{equation}
with $g,\varphi_1\,\dots,\varphi_N$ and $N$ as in \eqref{giyfci}.
Let  $n\in\mathbb N$ be such that the support of each $\varphi_i$
($i=1,\dots,N$) is a subset of $[1/n,\infty)\times X$. Then the function $F$ is $\mathcal B_n(\K(X))$-measurable.

\begin{theorem}\label{tuyr75e464w} Assume that \eqref{guf7r}--\eqref{uyfr7r78r}  hold.
Assume that the function $l$ is continuously differentiable on $\Rp$ and $l'\in L^1(\Rp,ds)$.  Assume that $F,G\in\FCK$ are measurable with respect to $\mathcal B_n(\K(X))$.
 Then, for each $(v,h)\in\mathfrak g$,
\begin{align}
\int_{\K(X)}(\nabla^{\M}_{(v,h)}F)(\eta)G(\eta)\,d\mu_\lambda(\eta)&=- \int_{\K(X)}F(\eta)(\nabla^{\M}_{(v,h)}G)(\eta)\,d\mu_\lambda(\eta)\notag\\
&\quad-\int_{\K(X)}F(\eta)G(\eta) B^{(n)}_{(v,h)}(\eta)\,d\mu_\lambda(\eta),\label{tyre675e4}\end{align}
where
 \begin{align}
 B^{(n)}_{(v,h)}&=B^{(n)}_v+B_h,\notag\\
 B^{(n)}_v(\eta)&=\sum_{x\in\tau(\eta):\, s_x\ge1/n}\operatorname{div}^X v(x),\notag\\
 B_h(\eta)&=\int_X \frac{l'(s_x)}{l(s_x)}\,
 h(x)\,d\eta(x)+l(0)\int_X h(x)\,dx .\label{huf7tyr}
 \end{align}
Here, $l(0):=\lim_{s\to0}l(s)$ and $\operatorname{div}^X v(x)$ denotes the divergence of $v(x)$ on $X$.
 \end{theorem}

\begin{proof} At least heuristically, formulas \eqref{tyre675e4}, \eqref{huf7tyr}
may be easily derived from Theorem~\ref{yur767} and its proof, see, in  particular,  \eqref{rtew6u4w6} and \eqref{hjvfut7r678} and compare with formulas \eqref{huvfrt7r75} and \eqref{uyfr7r7r}. In fact, for some measures $\mu_\lambda$, like for example the gamma measure, one may rigorously justify these calculations. However, in the general case, such a justification seems to be quite a difficult problem. So below, we will present an alternative proof, which is based on the Mecke formula for the Poisson measure  \cite[Satz~3.1]{Mecke}, see also \cite[Exercise~11.1]{Kal}.

Using the Mecke formula satisfied by the Poisson measure $\pi_\varkappa$ and the measurable bijective mapping $\mathscr R  :\Gamma_{pf}(\hat X)\to\mathbb K(X)$, we conclude that, for each
measurable function $G:\K(X)\times \hat X\to[0,\infty]$,
\begin{equation}\label{yfuftyttfyoi9u}
\int_{\K(X)}d\mu_\lambda(\eta)\int_X d\eta(x)\,G(\eta,s_x,x)=
\int_{\K(X)}d\mu_\lambda(\eta)\int_{X}dx\int_{\Rp}\,ds\, l(s) G(\eta+s\delta_x,s,x).
\end{equation}

Let $F:\K(X)\to\R$, $\eta\in\K(X)$, and $x\in\tau(\eta)$. We define
\begin{align}
(\nabla^X_xF)(\eta):=&\nabla^X_y\big|_{y=x}F(\eta-s_x\delta_x+s_x\delta_y)\label{uyfruf},\\
(\nabla^{\Rp}_{x})F(\eta):=&s_x\frac{d}{du}\Big|_{u=s_x}F(\eta-s_x\delta_x+u\delta_x),\label{fr8r}
\end{align}
provided the derivatives  exist.
Here the variable $y$ is from $X$,  $\nabla^X_y$ denotes the  gradient on $X$ in the $y$ variable, and
the variable $u$ is from $\Rp$.
An easy calculation shows that, for each function $F\in\FCK$ and $(v,h)\in\mathfrak g$,
\begin{align}
(\nabla_v^{\M}F)(\eta)&=
\int_X \langle(\nabla_x^{X}F)(\eta),v(x)\rangle_{T_x(X)}\,d\tilde \eta(x)=
\sum_{x\in\tau(\eta)}\langle(\nabla_x^{X}F)(\eta),v(x)\rangle_{T_x(X)},\notag\\
(\nabla_h^{\M}F)(\eta)&=
\int_X (\nabla_x^{\Rp}F)(\eta)h(x)\,d\tilde \eta(x)=
\sum_{x\in\tau(\eta)}
(\nabla_x^{\Rp}F)(\eta)h(x).
\label{yfdstre64u}
\end{align}
Here, $T_x(X)$ denotes the tangent space to $X$ at point $x$.

By  \eqref{yfuftyttfyoi9u} and \eqref{yfdstre64u},
\begin{align}
& \int_{\K(X)}(\nabla^{\M}_{(v,h)}F)(\eta)G(\eta)\,d\mu_\lambda(\eta)\notag\\
&\quad =\int_{\K(X)}d\mu_\lambda(\eta)\int_{\Rp} d\lambda(s)\int_{ X}dx\,\la\nabla^X_x F(\eta+s\delta_x), v(x)\ra_{T_x(X)} \,G(\eta+s\delta_x)\notag\\
&\qquad\text{}+\int_{\K(X)}d\mu_\lambda(\eta)\int_{ X}dx\int_{\Rp} ds\, l(s)\left(\frac{d}{ds}\, F(\eta+s\delta_x)\right)h(x)G(\eta+s\delta_x).\label{gf7yr}
\end{align}
Note that, since the function $l$ is continuously differentiable on $\Rp$ and $l'$ is integrable,
we get $\lim_{s\to\infty}l(s)=0$ and $\lim_{l\to 0}l(s)=-\int_{\Rp}l'(s)\,ds=l(0)$. By the definition of $\FCK$, for any fixed $\eta\in\K(X)$ and $x\in X$, the function $\Rp\ni s\mapsto F(\eta+s\delta_x)$ is bounded, smooth, and its derivative has a compact support in $\Rp$. Furthermore, for any fixed $\eta\in\K(X)$ and $s\in\Rp$, the function $X\ni x\mapsto F(\eta+s\delta_x)$ is smooth and its gradient is identically equal to zero if $s<\frac1n$. Hence, integration by parts in \eqref{gf7yr} gives
\begin{align}
& \int_{\K(X)}(\nabla^{\M}_{(v,h)}F)(\eta)G(\eta)\,d\mu_\lambda(\eta)\notag\\
&\quad =
\int_{\K(X)}d\mu_\lambda(\eta)\int_{[\frac1n,\infty)} d\lambda(s)\int_{ X}dx\,
F(\eta+s\delta_x)\notag\\
&\qquad\qquad\times \big(
-\la\nabla^X_x G(\eta+s\delta_x), v(x)\ra_{T_x(X)}-G(\eta+s\delta_x)\operatorname{div}^Xv(x)\big)
\notag\\
&\qquad\text{}+\int_{\K(X)}d\mu_\lambda(\eta)\int_{ X}dx\,h(x)\int_{\Rp} ds\, l(s)F(\eta+s\delta_x)\left(-\frac{d}{ds}\, G(\eta+s\delta_x)-G(\eta)\frac{l'(s)}{l(s)}\right)\notag\\
&\qquad\text{}-\int_{\K(X)}d\mu_\lambda(\eta)\int_{ X}dx\,h(x) l(0)F(\eta)G(\eta).\label{frte6u}
\end{align}
Applying \eqref{yfuftyttfyoi9u} to \eqref{frte6u}, we get the statement.
\end{proof}

\begin{example}
For the gamma measure,
$$ B_h(\eta)=-\langle h,\eta\rangle+\int_X h(x)\,dx.$$
If $l(s)$ satisfies \eqref{utyfr65e6afxt} (with $\alpha\in(0,1]$), we get  $l(0)=0$ and
$$\frac{l'(s)}{l(s)}=-\frac\alpha{s\log s},\quad s\in(0,\rho).$$
\end{example}

\section{Laplace operator}

Our next aim is to construct a Laplace operator associated with the measure $\mu_\lambda$. The definition of such an operator depends on the choice of a  tangent bundle.

Recall that we constructed the measure $\mu_\lambda$ by taking the pushforward of the Poisson measure $\pi_\varkappa$ under  the mapping $\mathscr R$. According to \cite{AKR}, a tangent space  to $\Gamma(\hat X)$ at $\gamma\in \Gamma(\hat X)$ is defined by
$$T_\gamma(\Gamma)=L^2(\hat X\to T(X)\times\R,\gamma).$$
Note that, for each $\gamma\in\Gamma_{pf}(\hat X)$,
$$T_\gamma(\Gamma)=\bigoplus_{(s,x)\in\gamma}(T_x(X)\times\R)
=\bigoplus_{x\in\tau(\mathscr R\gamma)}(T_x(X)\times\R).$$
So, it is natural to introduce a tangent space to $\K(X)$ at $\eta\in\K(X)$ by
\begin{equation}\label{rte75i6}
T_\eta(\K):=\bigoplus_{x\in\tau(\eta)}(T_x(X)\times\R)=L^2(X\to T(X)\times\R,\tilde\eta)=
L^2(X\to T(X),\tilde\eta)\oplus L^2(X,\tilde\eta).\end{equation}
 We then define a gradient of a differentiable function $F:\K(X)\to\R$ at $\eta$ as the  element $(\nabla^{\K}F)(\eta)$ of $T_\eta(\K)$ which satisfies
$$(\nabla^{\M}_{(v,h)}F)(\eta)=\langle (\nabla^{\K}F)(\eta),(v,h)\rangle_{T_\eta(\K)}\quad\text{for all }(v,h)\in\mathfrak g.$$
Note that, by \eqref{yfdstre64u}, for each $F\in\FCK$,
\begin{equation}\label{xrdtr}
(\nabla^\K F)(\eta,x)=\big((\nabla_x^XF)(\eta),(\nabla_x^{\Rp}F)(\eta)\big),\quad\eta\in\K(X),\  x\in\tau(\eta). \end{equation}

Let us assume that conditions \eqref{vytd7}, \eqref{tyd6e6i} are satisfied. We consider the {\it Dirichlet integral} (or the {\it Dirichlet form})
\begin{equation}\label{gddydyyd}
\mathcal E_\lambda^{\mathbb K}(F,G):=\frac12\int_{\K(X)}\langle
\nabla^{\mathbb K}F,\nabla^{\mathbb K}G\rangle_{T(\K)}\,d\mu_\lambda,\quad F,G\in\FCK.\end{equation}
It can be easily seen from \eqref{yfuftyttfyoi9u} and \eqref{xrdtr} that the function under the sign of integral in \eqref{gddydyyd} is indeed integrable. Furthermore, $\mathcal E_\lambda^{\mathbb K}$ is a well defined, symmetric bilinear form on $L^2(\K(X),\mu_\lambda)$.

 For a function $F\in\FCK$, $\eta\in\K(X)$, and $x\in\tau(\eta)$, we denote
 \begin{align}
 (\Delta^X_x F)(\eta):&=\Delta_y^X\big|_{y=x}F(\eta-s_x\delta_x+s_x\delta_y),\label{out678r5}\\
 (\Delta^{\Rp}_xF)(\eta):&=\Delta^{\Rp}_u\big|_{u=s_x}
 F(\eta-s_x\delta_x+u\delta_x).\label{uiyt6785o}
 \end{align}
 Here, for a twice differentiable function $f:\Rp\to\R$,
\begin{equation}\label{fctyr567}
(\Delta^{\Rp}f)(s):=s^2f''(s)+sf'(s)+s^2\,\frac{l'(s)}{l(s)}\,f'(s),\quad s\in\Rp,\end{equation}
 and $\Delta^X=\operatorname{div}^X\nabla^X$ is the Laplace--Beltrami operator on $X$.

 \begin{theorem}\label{vgfyt7r}
 Assume that \eqref{vytd7}--\eqref{uyfr7r78r} hold. Assume that the function $l$ is continuously differentiable on $\Rp$.
 For each $F\in\FCK$, we define
\begin{equation}\label{ftur766}
 (L^{\K}_\lambda F)(\eta):=\frac12\int_X \left[(\Delta^X_x F)(\eta)+(\Delta^{\Rp}_xF)(\eta)\right]\,d\tilde\eta(x),\quad \eta\in\K(X). \end{equation}
 Then $(L_\lambda^\K,\FCK)$ is a symmetric operator in $L^2(\K(X),\mu_\lambda)$ which satisfies
 \begin{equation}\label{rte6e6}
\mathcal E_\lambda^{\mathbb K}(F,G)=(-L_\lambda^{\mathbb K}F,G)_{L^2(\K(X),\,\mu_\lambda)},\quad F,G\in\FCK.\end{equation}
The bilinear form $(\mathcal E_\lambda^{\mathbb K},\FCK)$ is closable on $L^2(\K(X),\mu_\lambda)$, and its closure is denoted by $(\mathcal E_\lambda^{\mathbb K},D(\mathcal E_\lambda^{\mathbb K}))$. The operator $(L_\lambda^{\mathbb K},\FCK)$ has Friedrichs' extension, denoted by $(L_\lambda^{\mathbb K},D(L_\lambda^{\mathbb K}))$---the generator of the closed symmetric form $(\mathcal E_\lambda^{\mathbb K},D(\mathcal E_\lambda^{\mathbb K}))$.
  \end{theorem}

  \begin{remark}Let $L^{\hat X}$ denote the operator acting on functions on $\hat X$ as follows:
\begin{equation}\label{fyd7yr867r}
 (L^{\hat X}f)(s,x):=\frac12(\Delta^X_xf)(s,x)+\frac12(\Delta^{\Rp}_sf)(s,x).
 \end{equation}
 Then,  the following informal formula holds:
 $$ (L_\lambda^{\K}F)\bigg(\sum_is_i\delta_{x_i}\bigg)=\sum_j
 L^{\hat X}_{(s_j,x_j)}F\bigg(\sum_is_i\delta_{x_i}\bigg),
 $$
 where $L^{\hat X}_{(s_j,x_j)}$ is the $L^{\hat X}$ operator acting in the $(s_j,x_j)$ variable.
 \end{remark}

 \begin{remark}Compared with the integration by parts formula
from Theorem~\ref{tuyr75e464w}, in the definition of the operator $L_\lambda^{\K}$ we do not use the cut-off condition
 $s_x\ge1/n$ for some $n$. This is actually due to the fact that, for a function $F\in\FCK$, we get, for some $n\in\N$,
$$\nabla^X_y F(\eta- s_x\delta_x+s_x\delta_y)=0 \quad \text{if }s_x<1/n.$$
Hence, if $s_x<1/n$,
$$\Delta^X_y F(\eta- s_x\delta_x+s_x\delta_y)=\operatorname{div}^X_y\nabla^X_y F(\eta- s_x\delta_x+s_x\delta_y)=0.$$
 Thus, although the integration by parts formula for the measure $\mu_\lambda$ holds only in a weak sense, we get a proper {\it Laplace operator $L_\lambda^{\K}$ relative to the measure $\mu_\lambda$}.
 \end{remark}

 \begin{proof}[Proof of Theorem~\ref{vgfyt7r}]
 Formulas \eqref{ftur766}, \eqref{rte6e6} can be derived from Theorem~\ref{tuyr75e464w}. Alternatively, we may give a direct proof of these formulas by analogy with the proof of Theorem~\ref{tuyr75e464w}. Indeed, by \eqref{yfuftyttfyoi9u}--\eqref{uyfruf},  and  \eqref{rte75i6}--\eqref{gddydyyd}, we get, for any $F,G\in\FCK$,
 \begin{align}
 \mathcal E^\K_\lambda(F,G)&=\frac12\int_{\K(X)}d\mu_\lambda(\eta)\int_{X}dx\int_{\Rp}ds\,\frac{l(s)}s
 \bigg[
 \langle\nabla^X_x F(\eta+s\delta_x),\nabla^X_x G(\eta+s\delta_x)\rangle_{T_x(X)}
  \notag\\
 &\qquad+\bigg(s\frac{d}{ds}F(\eta+s\delta_x)\bigg)\bigg(s\frac{d}{ds}F(\eta+s\delta_x)\bigg)
 \bigg].\label{yur75i6erdes}
\end{align}
  From here, using integration by parts and  \eqref{yfuftyttfyoi9u}, formulas \eqref{ftur766}, \eqref{rte6e6} follow.

  Let us show that, for each $F\in\FCK$, $L_\lambda^\K F\in L^2(\K(X),\mu_\lambda)$. It follows from the definition of $\FCK$, formulas \eqref{uiyt6785o}--\eqref{fctyr567}, and the assumption of the theorem that, for each $F\in\FCK$, there exist a compact set $\Lambda\subset \hat X$ and a constant $C_1>0$ such that
 $$
|(\Delta_x^{X}F)(\eta)|+ |(\Delta_x^{\Rp}F)(\eta)|\le C_1\chi_\Lambda(s_x,x),\quad \eta\in\K(X),\ x\in\tau(\eta),
 $$
 where $\chi_\Lambda$ denotes the indicator function of $\Lambda$. Thus, by \eqref{ftur766}, it suffices to show that
 $$\int_{\K(X)}\bigg(\int_X\chi_\Lambda(s_x,x)\,d\tilde\eta(x)\bigg)^2\,d\mu_\lambda(\eta)<\infty$$
 This can be easily deduced from \eqref{yfuftyttfyoi9u}.

 The statements that the bilinear form $(\mathcal E_\lambda^{\mathbb K},\FCK)$ is closable on $L^2(\K(X),\mu_\lambda)$ and that the operator $(L_\lambda^{\mathbb K},\FCK)$ has Friedrichs' extension are now standard, see e.g.\ \cite[Theorem~X.23]{RS}.
 \end{proof}

 \begin{theorem}\label{ytre654e6u}
 Let the assumptions of Theorem~\ref{vgfyt7r} be satisfied.
 Then the operator $(L_\lambda^{\mathbb K},D(L_\lambda^{\mathbb K}))$ is essentially self-adjoint on $\FCK$.
 \end{theorem}

 \begin{proof}
 Consider the symmetric operator $(\frac12\Delta^{\Rp},C_0^\infty(\Rp))$. We construct the unitary operator
 \begin{gather}
  \textstyle U: L^2(\Rp,\frac{l(s)}s\,ds)\to L^2(\R,l(e^u)\,du),\notag\\
  (Uf)(u)=f(e^u),\quad u\in\R.\label{bufr7}
  \end{gather}
Then $UC_0^\infty(\Rp)=C_0^\infty(\R)$ and for any $g\in C_0^\infty(\R)$
\begin{equation}\label{xtst5w}
 (L^{\R}\,g)(u):=(U\frac12\Delta^{\Rp} U^{-1}g)(u)=\frac12\,g''(u)+\frac12\bigg(\frac{d}{du}\log(l(e^u))\bigg) g'(u),\quad u\in\R.\end{equation}
  Hence, by \cite[Theorem~2.3]{W}, the operator $(\frac12\Delta^{\Rp},C_0^\infty(\Rp))$ is essentially self-adjoint in $L^2(\Rp,\frac{l(s)}s\,ds)$. Furthermore, it is well known that the symmetric operator $(\frac12\Delta^X,C_0^\infty(X))$ is essentially self-adjoint in $L^2(X,dx)$. Therefore, the operator $(L^{\hat X},C_0^\infty(\hat X))$, defined by  \eqref{fyd7yr867r}, is essentially self-adjoint in $L^2(\hat X,\varkappa)$.

 For a real separable  Hilbert space $\mathcal H$, we denote
by $\mathcal F(\mathcal H)$ the symmetric Fock space over $\mathcal H$:
$$ \mathcal F(\mathcal H):=\bigoplus_{n=0}^\infty\mathcal H^{\odot n}n!\,. $$
Here $\odot$ stands for symmetric tensor product. Let $(\mathscr A,\mathscr D)$ be a densely defined
symmetric operator in $\mathcal H$.
We denote by $\mathcal F_{\mathrm{alg}}(\mathscr D)$
the subset of  $\mathcal F(\mathcal H)$ which is the linear span of the vacuum vector $\Psi=(1,0,0,\dots)$ and vectors of the form
$\varphi_1\odot \varphi_2\odot\dots\odot \varphi_n$, where $\varphi_1,\dots,\varphi_n\in\mathscr D$ and $n\in\N$. The differential second quantization $d\operatorname{Exp}(\mathscr A)$ is defined as the symmetric operator in $\mathcal F(\mathcal H)$ with domain $\mathcal F_{\mathrm{alg}}(\mathscr D)$ which acts as follows:
\begin{gather}
d\operatorname{Exp}(\mathscr A)\Psi:=0,\notag\\
d\operatorname{Exp}(\mathscr A)\varphi_1\odot \varphi_2\odot\dots\odot \varphi_n:=\sum_{i=1}^n \varphi_1\odot \varphi_2\odot\dots\odot
(\mathscr A\varphi_i)\odot\dots\odot \varphi_n.\label{uyrde56e6}
\end{gather}
By e.g.\ \cite[Chap.~6, subsec. 1.1]{BK}, if the operator $(\mathscr A,\mathscr D)$ is essentially self-adjoint in $\mathcal H$, then the differential second quantization $(d\operatorname{Exp}(\mathscr A), \mathcal F_{\mathrm{alg}}(\mathscr D))$ is essentially self-adjoint in $\mathcal F(\mathcal H)$. Hence, $(d\operatorname{Exp}(L^{\hat X}), \mathcal F_{\mathrm{alg}}(C_0^\infty(\hat X)))$
is essentially self-adjoint in $\mathcal F(L^2(\hat X,\varkappa))$.

  Let
 \begin{equation}\label{gtyr756e7i}
 I:L^2(\Gamma(\hat X),\pi_\varkappa)\to \mathcal F(L^2(\hat X,\varkappa))
 \end{equation}
  denote the unitary operator which is derived through multiple stochastic integrals with respect to the centered Poisson random measure with intensity measure $\varkappa$, see e.g.\ \cite{Surgailis}.
Let $\mathscr P$ denote the set of functions on $\Gamma(\hat X)$ which are finite sums of $\la\varphi_1,\cdot\ra\dotsm \la\varphi_n,\cdot\ra$ with $\varphi_1,\dots,\varphi_n\in C_0^\infty(\hat X)$, $n\in\N$, and constants. Thus, $\mathscr P$ is a set of polynomials on $\Gamma(\hat X)$. Using the properties of $I$, one shows that
$$ I^{-1}\mathcal F_{\mathrm{alg}}(C_0^\infty(\hat X))=\mathscr P.$$

For each $(s,x)\in\hat X$, we define an annihilation operator at $(s,x)$ acting on\linebreak  $\mathcal F_{\mathrm{alg}}(C_0^\infty(\hat X))$ by the formula
\begin{equation*}\partial_{(s,x)}\Psi:=0,\quad \partial_{(x,s)}\varphi_1\odot\varphi_2\odot\dots\odot\varphi_n:=
\sum_{i=1}^n \varphi_i(s,x)\varphi_1\odot\varphi_2\odot\dots\odot\check\varphi_i\odot\dots\odot\varphi_n,\end{equation*}
where $\check\varphi_i$ denotes  absence of $\varphi_i$. We will preserve the notation $\partial_{(s,x)}$ for the operator $I\partial_{(s,x)}I^{-1}:\mathscr P\to\mathscr P$.
This operator admits the following explicit representation:
\begin{equation}\label{vcyd6i5}
\partial_{(s,x)}F(\gamma)=F(\gamma+\delta_{(s,x)})-F(\gamma)\end{equation}
for $\pi_\varkappa$-a.a.\ $\gamma\in\Gamma(\hat X)$, see e.g.\
\cite{IK,NV}.

Denote $\mathscr L:=I^{-1}d\operatorname{Exp}(L^{\hat X})I$.
Then $(\mathscr L,\mathscr P)$ is the generator of the bilinear form
\begin{align}
\mathscr E(F,G)&=\frac12\int_{\Gamma(\hat X)}d\pi_\varkappa(\gamma)
\int_{\hat X}d\varkappa(s,x)\bigg[\la \nabla^X_x\partial_{(s,x)}F(\gamma),\nabla^X_x\partial_{(s,x)}G(\gamma)\ra_{T_x(X)}\notag\\
&\qquad+\bigg(s\frac{d}{ds}\partial_{(s,x)}F(\gamma)\bigg)\bigg(s\frac{d}{ds}\partial_{(s,x)}G(\gamma)\bigg)\bigg].\label{ft6er557}
\end{align}
Note that, by \eqref{vcyd6i5},
\begin{equation}\label{uyfr756ir}
\nabla^X_x\partial_{(s,x)}F(\gamma)=\nabla_x^X F(\gamma+\delta_{(s,x)}),\quad \frac{d}{ds}\partial_{(s,x)}F(\gamma)=\frac{d}{ds}F(\gamma+\delta_{(s,x)}).
\end{equation}
Since $(\mathscr L,\mathscr P)$ is essentially self-adjoint in $L^2(\Gamma(\hat X),\pi_\varkappa)$, by \eqref{yur75i6erdes}, \eqref{ft6er557}, and \eqref{uyfr756ir}, to prove the theorem, it suffices to show that, for any polynomial $p:\R^N\to\R$ of $N$ variables, and any $\varphi_1,\dots,\varphi_N\in C_0^\infty$, the function
$$F(\eta)=p\big(\langle\!\langle \varphi_1,\eta\rangle\!\rangle,\dots,\langle\!\langle \varphi_N,\eta\rangle\!\rangle\big),\quad\eta\in\K(X), $$
belongs to the closure of the symmetric operator $(L_\lambda^{\K},\FCK)$ in $L^2(\K(X),\mu_\lambda)$ (compare with \eqref{tdrytd6}, \eqref{fty6ed6u45}). But this can be easily done by approximation.
 \end{proof}

 Let us recall the notion of a second quantization in a symmetric Fock space. Let $B$ be a bounded linear operator in a real separable Hilbert space $\mathcal H$. Assume that the norm of $B$ is $\le1$.
 One defines the second quantization of $B$ the bounded linear operator $\operatorname{Exp}(B)$ in $\mathcal F(\mathcal H)$ which satisfies $\operatorname{Exp}(B)\Psi:=\Psi$ and for each $n\in\N$, the restriction of $\operatorname{Exp}(B)$ to $\mathcal H^{\odot n}$ coincides with $B^{\otimes n}$.

Recall the unitary operator $I$, see \eqref{gtyr756e7i}. In view of the  mapping $\mathscr R$, we can equivalently treat $I$ as a unitary operator
$$I:L^2(\K( X),\mu_\lambda)\to \mathcal F(L^2(\hat X,\varkappa)). $$

\begin{corollary}\label{tfi7r79}
 Let the assumptions of Theorem~\ref{vgfyt7r} be satisfied.
Then
$$ I^{-1}\exp(tL_\lambda^\K)I^{-1}=\operatorname{Exp}\big(\exp(tL^{\hat X})\big),\quad t\ge0.$$
Here $(L^{\hat X},D(L^{\hat X}))$ is the self-adjoint operator in $L^2(\hat X,\varkappa)$ defined as the closure of $(L^{\hat X}, C_0^\infty(\hat X))$, see \eqref{fyd7yr867r}.
\end{corollary}

\begin{proof}
The result follows from the proof of Theorem~\ref{ytre654e6u} and the properties of a second quantization  (cf.\ e.g.\ \cite[Chap.~6, subsec.~1.1]{BK}).
\end{proof}

\section{Diffusion processes}
Let us assume that the dimension of the manifold $X$ is $\ge2$. By using  the theory of Dirichlet forms \cite{MR,MRpaper}, it can be shown   \cite{CKL} that there exists
a conservative diffusion process on $\K(X)$ (i.e., a conservative strong Markov process with continuous sample paths in $\K(X)$) which has $\mu_\lambda$
as its symmetrizing measure and its $L^2(\K(X),\mu_\lambda)$-generator is
$(L_\lambda^{\mathbb K},D(L_\lambda^{\mathbb K}))$. Unfortunately, the theory of Dirichlet forms gives rather little information apart from the very existence of the   process.
In the following subsection, under a little bit stronger assumptions on the manifold $X$ and the function $l$, we will present an explicit construction of (a version of)  this Markov process.  To this end, we will use ideas from
\cite{KLR2}.

\subsection{Explicit construction of the process}

We introduce a metric $d_\lambda$ on $\mathbb R_+$ which is associated with the measure $\lambda$: for any $s_1,s_2\in\R_+$ with $s_1<s_2$, we set
$$d_\lambda(s_1,s_2)=d_\lambda(s_2,s_1): =\lambda((s_1,s_2)).$$
We then define a metric on $\hat X$ by
\begin{equation}\label{hjgufutr}
d_{\hat X}((s_1,x_1),(s_2,x_2)):=\max\big\{d_\lambda(s_1,s_2),d_X(x_1,x_2)\big\},\end{equation}
where $d_X$ is the Riemannian metric on $X$. We fix a point $x_0\in X$ and denote by $B_{\hat X}(r)$ an open ball in $\hat X$ which is centered at $(1,x_0)$ and of radius $r$ (with respect to the metric $d_{\hat X}$).

We define the following measurable subset of $\K(X)$:
\begin{multline}
\Theta:= \big\{
\eta\in\K(X)\mid
|\tau(\eta)|=\infty\\
\text{and }
\exists K\in\N\  \forall r\in\N:\  |(\mathscr R^{-1}\eta)\cap B_{\hat X}(r)|\le K\varkappa\big(B_{\hat X}(r)\big)\big\}.\label{gudr6e7ir}\end{multline}
(Recall that the measure $\varkappa$ on $\hat X$  and the mapping $\mathscr R$ are  defined  by \eqref{5}  and \eqref{uyfgutfrrsesa}, respectively.)
It follows from the explicit construction of the measure $\mu_\lambda$ in subsec.~\ref{syigufuq} and  e.g.\ \cite{NZ} that $\mu_\lambda(\Theta)=1$.
We denote by $\mathcal B(\Theta)$ the trace $\sigma$-algebra of $\mathcal B(\K(X))$ on $\Theta$. Thus, we may consider $\mu_\lambda$ as a probability measure on $(\Theta,\mathcal B(\Theta))$. We also equip $\Theta$ with the topology induced by the topology on $\K(X)$. So, our aim now is to construct a continuous Markov process on $\Theta$ with generator $L_\lambda^{\mathbb K}$.

About the function $l$ we will  assume  below that
\begin{gather}
l\in C^2(\R_+),\label{yu7re}\\
l'\in L^1(\R_+,ds),\label{hy890y798}\\
\sup_{s\in\R_+}\frac{l'(s)s}{l(s)}<\infty,\label{tyr65r}\\
\sup_{s\in(0,1)}\frac{l'(s)s}{l(s)\log(s)}<\infty.\label{tye64u3e}
\end{gather}
One can easily check that these conditions are satisfied for the functions $l$ from Examples~\ref{example3} and~\ref{tyre56i54ei}.

Let us consider the following stochastic differential equation on $\R$:
\begin{equation}\label{ydr6ed64}
dY(t)=dW(t)+\frac{l'(e^{Y(t)})e^{Y(t)}}{2\,l(e^{Y(t)})}\,dt\end{equation}
with initial condition $Y(0)=y_0$.  Here $W(t)$ is a Brownian motion on $\R$. Note that \eqref{tyr65r} and \eqref{tye64u3e} imply  existence of a constant $C_2>0$ such that
$$\frac{s\,l'(e^s)e^s}{2\,l(e^s)}\le C_2(1+s^2),\quad s\in\R.$$
Hence, by Theorem 3 and Remark 3 in Section~6 of \cite{GS}, the stochastic differential equation \eqref{ydr6ed64} has a unique strong solution. As follows from  the proof of Theorem~ \ref{ytre654e6u}, the operator $L^{\R}$, defined by formula \eqref{xtst5w}, is essentially self-adjoint on $C^\infty_0(\R)$ in $L^2(\R,l(e^s)\,ds)$. Denote by $(L^{\R},D(L^{\R}))$ the closure of this operator.
Then, by using e.g.\ Chapter~1 of \cite{Eberle}, we conclude that the conservative Markov process $Y=(Y(t))_{t\ge0}$ has $l(e^s)\,ds$ as symmetrizing measure and $(L^{\R},D(L^{\R}))$ is its $L^2$-generator. Hence, by \eqref{bufr7} and \eqref{xtst5w}, the conservative Markov process $Z=(Z(t))_{t\ge0}$ with $Z(t):=e^{Y(t)}$  has $\lambda$ as symmetrizing measure and $(\frac12\Delta^{\R_+},D(\Delta^{\R_+}))$ is its $L^2$-generator. Here $(\frac12\Delta^{\R_+},D(\Delta^{\R_+}))$ is the closure of the operator $(\frac12\Delta^{\Rp},C_0^\infty(\R_+))$ in $L^2(\R_+,\lambda)$.

\begin{theorem}\label{yre6e}
Assume that the function $l$ satisfies \eqref{yu7re}--\eqref{tye64u3e}.
Let the dimension of the manifold $X$ be $\ge2$. Furthermore, assume that $X$ satisfies the following conditions:

\begin{itemize}
\item[(C1)] There exist $m\in\N$ and $C_3\in\R$ such that, for all $r>0$ and $\beta\ge1$
$$\int_{B_X(\beta r)}dx\le C_3\beta^m\int_{B_X(r)}dx.$$
Here $B_X(r)$ denotes the open ball in $X$ of radius $r$, centered at $x_0$.

\item[(C2)] The manifold $X$ is stochastically complete, i.e., for any $x_0\in X$, the Brownian motion $(B_t)_{t\ge0}$ on $X$ starting at $x_0$ has infinite lifetime.

\item[(C3)] The heat kernel $p(t,x,y)$ on $X$  satisfies the Gaussian upper bound for small values of $t$:

$$p(t,x,y)\le C_4\, t^{-n/2}\exp\left[-\frac{d_X(x,y)^2}{Dt}\right],\quad t\in(0,\varepsilon],\ x,y\in X,$$
where $n\in\N$, $\varepsilon>0$ and $C_4$ and $D$ are positive constants.

\end{itemize}

Then the following statements hold.

(i) For any $\eta=\sum_{i=1}^\infty s_i\delta_{x_i}\in\Theta$, let $(Z_i(t))_{t\ge0}$ and $(B_i(t))_{t\ge0}$, $i\in\N$, be independent stochastic processes such that
$Z_i(t)=e^{Y_i(t)}$, where $Y_i(t)$ is the strong solution of the stochastic differential equation \eqref{ydr6ed64} with initial condition $Y_i(0)=\ln(s_i)$ and $B_i(t)$ is a Brownian motion on $X$ with initial condition $B_i(0)=x_i$. For $t\ge0$, denote
$$\mathfrak X(t):=\sum_{i=1}^\infty Z_i(t)\delta_{B_i(t)}.$$
In particular, $\mathfrak X(0)=\eta$.
Then, with probability one, $\mathfrak X(t)\in\Theta$ for all $t\ge0$ and the sample path $[0,\infty)\ni t\mapsto\mathfrak X(t)\in\Theta$ is continuous.

(ii) Denote $\Omega:=C([0,\infty)\to\Theta)$ and  let $\mathcal F$ be the corresponding cylinder $\sigma$-algebra on $\Omega$. For each $\eta\in\Theta$, denote by $\mathbb P_\eta$ the probability measure on $(\Omega,\mathcal F)$ which is the law of the stochastic process $(\mathfrak X(t))_{t\ge0}$ from (i) starting at $\eta$. Assume now that $\mathfrak X(t)$ is chosen canonically, i.e., for each $t\ge0$ we have $\mathfrak X(t):\Omega\to\Theta$, $\mathfrak X(t)(\omega)=\omega(t)$. Furthermore, for each $t\ge0$, denote $\mathcal F_t:=\sigma\{\mathfrak X(u),\, 0\le u\le t\}$ and let $(\zeta_t)_{t\ge0}$ be the natural time shifts: $\zeta_t:\Theta\to\Theta$, $(\zeta_t\omega)(u):=\omega(t+u)$. Then
$$M=(\Omega,\mathcal F,(\mathcal F_t)_{t\ge0}, (\zeta_t)_{t\ge0}, (\mathfrak X(t))_{t\ge0},(\mathbb P_\eta)_{\eta\in\Theta})$$
is a time homogeneous Markov process on the state space $(\Theta,\mathcal B(\Theta))$ with continuous paths and  transition probabilities $(\mathbb P_t(\eta,\cdot))_{t\ge0,\, \eta\in\Theta}$, where $\mathbb P_t(\eta,\cdot)$
is the distribution of $\mathfrak X(t)$ under $\mathbb P_\eta$.

(iii) For each $t>0$ and $F\in L^2(\Theta,\mu_\lambda)$, the function
$$\Theta\ni\eta\mapsto\int_\Theta F(\xi)\mathbb P_t(\eta,d\xi)$$
is a $\mu_\lambda$-version of  $e^{tL_\lambda^\K}F\in L^2(\Theta,\mu_\lambda)$.

\end{theorem}

\begin{remark}
If $X$ has a nonnegative Ricci curvature,
condition (C1) is satisfied with $C_3=1$ and $m$ being equal to the dimension of $X$, see e.g.\ \cite[Proposition~5.5.1]{D}.
\end{remark}

\begin{proof} (i) We divide the proof of this statement into several steps.

Step 1.  For $x\in X$, we denote by $ P_x$ the law of the Brownian motion $(B(t))_{t\ge 0}$ starting at $x$, and for $t>0$  and $A\in\mathcal B(X)$, we denote
$$p_t(x,A):=\int_A p(t,x,y)\,dy,$$
the transition probabilities of the Brownian motion.
For $A\subset X$, we denote by $T(A)$ the hitting time of $A$ by the Brownian motion.
 By \cite[Lemma~1]{KLR}, or \cite[Appendix A, Lemma~4]{Nelson} in the special case $X=\R$, we have for any $x\in X$ and  $r>0$,
\begin{equation}\label{tye656ir}
 P_x (T(B_X(x,r)^c)\le\varepsilon)\le 2 \sup_{t\in(0,\varepsilon]}\sup_{y\in X}p_t(y,B_X(y,r/4)^c).
\end{equation}
Here and below $B_X(x,r)$ denotes the open ball in $ X$, centered at $x$ and of radius $r$, and the index $c$ over a set denotes taking the compliment of this set.
By \cite[Lemma~8.2]{KLR}, (C3) implies existence of $C_5>0$ such that
\begin{equation}\label{bhvufu}
\sup_{t\in(0,\varepsilon]} \sup_{y\in X}p_t(y,B_X(y,r)^c)\le C_5 e^{-r},\quad r>0,
\end{equation}
where $\varepsilon$ is as in from (C3).
By \eqref{tye656ir} and \eqref{bhvufu}, for any  $\delta>0$ and $\alpha>0$,
\begin{equation}\label{tye7i564}
\sum_{n=1}^\infty\sup_{x\in X} P_x(T(B_X(x,\delta n^\alpha)^c)\le\varepsilon)<\infty.
\end{equation}

Step 2. For $s\in\R_+$ and $r>0$, we denote
 $$  R(s,r):=\{u\in\R_+\mid u>s,\, d_\lambda(s,u)\ge r\}.$$
Note that this set may be empty.
Let $P_s$ denote the law of $(Z(t))_{t\ge0}$ starting at $s$.
We denote by $T(R(s,r))$ the hitting time of $R(s,r)$ by $(Z(t))_{t\ge0}$.
(In the case where the set $R(s,r)$ is empty, we set $T(R(s,r)):=+\infty$).
We state that, for any $\delta>0$ and $\alpha>0$,
\begin{equation}\label{gtye65e}
\sum_{n=1}^\infty \sup_{s\in\R_+}P_s(T(R(s,\delta n^\alpha))\le\varepsilon)<\infty.
\end{equation}

Indeed, denote by $C_6$ the value of the supremum in \eqref{tyr65r}. By \cite[Chap.~VI,  Sec.~1]{IkedaWatanabe}, the stochastic process $(Y(t))_{t\ge0}$ solving the stochastic differential equation \eqref{ydr6ed64} with initial condition $Y(0)=y_0$ satisfies with probability one:
\begin{equation}\label{iut87p}
Y(t)\le W(t)+ y_0+C_6t,\quad t\ge0\end{equation}
 where $W(0)=0$.

As noted in the proof of Theorem~\ref{tuyr75e464w},  \eqref{yu7re} and \eqref{hy890y798} imply that the function $l$ has a finite limit at zero, hence it is bounded on $(0,1]$ by a constant $C_7>0$. Hence, for any $-\infty<z_1<z_2\le0$,
\begin{equation}\label{yut86o}
d_\lambda(e^{z_1},e^{z_2})=\int_{e^{z_1}}^{e^{z_2}}\frac {l(s)}s\,ds=\int_{z_1}^{z_2}l(e^{z})\,dz\le C_7(z_2-z_1). \end{equation}
 By \eqref{iut87p} and \eqref{yut86o}, formula \eqref{gtye65e} follows from the formula \eqref{tye7i564} applied to the case $X=\R$, $B(t)=W(t)$.

Step 3. By \eqref{hjgufutr}, for each $r>0$,
\begin{equation}\label{uf867t6}
B_{\hat X}(r)=B_\lambda(r)\times B_X(r),\end{equation}
where $B_\lambda(r)$ is the open ball in $\R_+$ with respect to metric $d_\lambda$, centered at 1 and of radius $r$. Hence, by (C1), for any $r>0$ and $\beta\ge1$,
\begin{equation}\label{vtuyr7o5}
\varkappa(B_{\hat X}(\beta r))\le C_3\beta^{m+1}\varkappa(B_{\hat X}(r)).\end{equation}

 Step 4. For $(s,x)\in\hat X$, we denote $|(s,x)|_{\hat X}:=d_{\hat X}((s,x),(1,x_0))$
and $|s|_\lambda:=d_\lambda(s,1)$, $|x|_X:=d_X(x,x_0)$.
Let $\eta=\sum_{i=1}^\infty s_i\delta_{x_i}\in\Theta$, and without loss of generality, we may assume that $|(s_{i+1},x_{i+1})|_{\hat X} \ge |(s_{i},x_{i})|_{\hat X}$ for all $i$. We define
\begin{equation}\label{bvgur7}
r(i):=\max\left\{n\in\N\mid n<\left(\frac{i}{KC_3\varkappa(B_{\hat X}(1))}\right)^{\frac1{m+1}}\right\},\end{equation}
where $K>0$ is the constant from \eqref{gudr6e7ir}, depending on $\eta$. (We assumed that $i$ is sufficiently large, so that the set on the right hand side of \eqref{bvgur7} is not empty.)
Then, by \eqref{vtuyr7o5},
\begin{align*}
|(\mathscr R^{-1}\eta)\cap B_{\hat X}(r(i))|&\le K\varkappa(B_{\hat X}(r(i)))\\&\le KC_3r(i)^{m+1}\varkappa(B_{\hat X}(1))<i.
\end{align*}
Hence, $(s_i,x_i)\not\in B_{\hat X}(r(i))$. Therefore, for all sufficiently large $i$,
\begin{equation}\label{yte675}
 \max\{|s_i|_\lambda,|x_i|_X\}=|(s_i,x_i)|_{\hat X}\ge r(i)\ge\left(\frac{i}{KC_3\varkappa(B_{\hat X}(1))}\right)^{\frac1{m+1}}-1\ge\delta i^{\frac1{m+1}}\end{equation}
for some $\delta>0$.

Step 5. Denote
$$ I:=\{i\in\N\mid |x_i|_X\ge |s_i|_\lambda\},\quad J:=\N\setminus I.$$
By \eqref{yte675}, for each $i\in I$ sufficiently large, we have  $|x_i|_X\ge \delta i^{\frac1{m+1}}$. Hence, by \eqref{tye7i564}, for some $C_8\ge0$,
\begin{equation}
\sum_{i\in I}P_{x_i}(T(B_X(x_i,|x_i|_X/2)^c)\le\varepsilon)\le C_8+\sum_{i\in I}P_{x_i}(T(B_X(x_i,\delta i^{\frac1{m+1}}/2)^c)\le\varepsilon)<\infty.\label{bhigt8ll}
\end{equation}
Hence, by the  Borel--Cantelli lemma, with probability one, for all but a finite number of $i\in I$,  we have $B_i(t)\in B_X(x_i,|x_i|_X/2)$ for all $t\in[0,\varepsilon]$, and so  $B_i(t)\not\in B_X(|x_i|_X/2)$ for all $t\in[0,\varepsilon]$.
This, in turn, implies that, with probability one, for all but a finite number of $i\in I$,  we have
\begin{equation}\label{ftdrs56}
(Z_i(t),B_i(t))\not\in B_{\hat X}(|(s_i,x_i)|_{\hat X}/2),\quad t\in[0,\varepsilon].\end{equation}

Analogously to \eqref{bhigt8ll}, using \eqref{gtye65e}, we get
$$\sum_{j\in J}P_{s_j}(T(R(s_j,|s_j|_\lambda/2))\le\varepsilon)<\infty. $$
Therefore, by the  Borel--Cantelli lemma, with probability one,  for all but a finite number of $j\in J$, the process $Z_j(t)$ does not reach the set $R(s_j,|s_j|_\lambda/2)$ for $t\in[0,\varepsilon]$, hence $Z_j(t)\not\in B_{\lambda}(|s_j|_\lambda/2)$.
Consequently, for all but a finite number of $j\in J$, formula \eqref{ftdrs56} holds with $i$ replaced by $j$.

Thus, with probability one, there exists a finite subset $\mathcal K\subset \N$ (depending on $\omega$) such that, for all $i\in\mathbb N\setminus\mathcal K$, formula \eqref{ftdrs56} holds. Let $k$ denote the number of the elements of $\mathcal K$, $k$ being a random variable.  Using \eqref{vtuyr7o5}, we conclude that, with probability one, for each $r\in\N$ and for all $t\in(0,\varepsilon]$,
\begin{align}
\big|\{(Z_i(t),B_i(t))\mid i\in\N\}\cap B_{\hat X}(r)\big|&\le \big|\{(s_i,x_i)\mid i\in\N\}\cap B_{\hat X}(2r)\big|+k\notag\\
&\le K\varkappa(B_{\hat X}(2r))+k\notag\\
&\le KC_3 2^{m+1}\varkappa(B_{\hat X}(r))+k\notag\\
&\le K'\varkappa(B_{\hat X}(r))\label{nohyiu9}
\end{align}
for some $K'>0$.

Step 6. Since the dimension of $X$ is $\ge2$, with probability one, $B_i(t)\ne B_j(t)$ for all $t\ge0$ and $i\ne j$, see e.g.\ (8.29) in \cite{KLR}.

Step 7. Consider any set $\{(u_i,y_i)\mid i\in\N\}\subset\hat X$
such that $y_i\ne y_j$ for $i\ne j$ and there exists a constant $K'>0$ for which
$$\big|\{(u_i,y_i)\mid i\in\N\}\cap B_{\hat X}(r)\big|\le K'\varkappa(B_{\hat X}(r)),\quad r\in\N.$$
We state: $\sum_{i=1}^\infty u_i\delta_{y_i}\in\Theta$.

Indeed, we only have to prove that $\sum_{i=1}^\infty u_i\delta_{y_i}\in\M(X)$. Without loss of generality, we may assume that
$|(u_{i+1},y_{i+1})|_{\hat X} \ge |(u_{i},y_{i})|_{\hat X}$ for all $i$. Just as in Step~4, we get
$  |(u_i,y_i)|_{\hat X}\ge \delta' i^{\frac1{m+1}}$ for all sufficiently large $i$. Here $\delta'>0$ depends on $K'$.

 Fix any compact $A\subset X$. Since $A$ is bounded, for all sufficiently large $i\in\N$ such that $y_i\in A$, we then have
$|u_i|_\lambda\ge \delta'i^{\frac1{m+1}}$. Hence, by \eqref{yut86o},
$$\delta' i^{\frac1{m+1}}\le C_7(-\log(u_i)),$$
and so
$$u_i\le\exp\bigg[-\frac{\delta'}{C_7}\,i^{\frac1{m+1}}\bigg].$$
Thus, $\sum_{i:\, y_i\in A}u_i<\infty$, which implies the statement.

Step 8. By Steps 5--7, with probability one, we have $\mathfrak X(t)=\sum_{i=1}^\infty Z_i(t)\delta_{B_i(t)}\in\Theta$ for all $t\in[0,\varepsilon]$. Furthermore, by the dominated convergence theorem, for each $f\in C_0(X)$, the mapping
$$[0,\varepsilon]\ni t\mapsto \langle f,\mathfrak X(t)\rangle=\sum_{i=1}^\infty Z_i(t)f(B_i(t))$$
is continuous with probability one. Therefore, the $\Theta$-valued stochastic process $(\mathfrak X(t))_{t\in[0,\varepsilon]}$ has a.s.\ continuous sample paths. By the Markov property, the statement (i) of the theorem immediately follows.

(ii) This statement immediately follows from the construction of the stochastic process $(\mathfrak X(t))_{t\ge0}$ and part (i) of the theorem.

(iii) This statement can be easily derived from Corollary~\ref{tfi7r79} analogously to the proof of Theorem~5.1 in \cite{KLR}, see also the proof of Theorem~2.2 in \cite{KLR2}.
\end{proof}

\subsection{An open problem: another diffusion process}

Let us recall that our definition of a tangent space $T_\eta(\K)$ at $\eta\in\K(X)$ was inspired by the  mapping \eqref{uyfgutfrrsesa} and the differential structure on the configuration space $\Gamma(\hat X)$. Alternatively, we may give a definition of a tangent space to $\M(X)$ at a generic Radon measure
$\eta\in\M(X)$. So, for each $\eta\in\M(X)$, we define
$$
T_\eta(\M):=L^2(X\to T(X)\times\R,\eta)=
 L^2(X\to T(X),\eta)\oplus L^2(X,\eta)$$
(compare with \eqref{rte75i6}). We then define a gradient of a differentiable function $F:\M(X)\to\R$ at $\eta\in\M(X)$ as the element $(\nabla^\M F)(\eta)$ of $T_\eta(\M)$ that satisfies
$$(\nabla^{\M}_{(v,h)}F)(\eta)=\langle (\nabla^{\M}F)(\eta),(v,g)\rangle_{T_\eta(\M )}\quad\text{for all }(v,h)\in\mathfrak g.$$
Then, analogously to \eqref{xrdtr}, we get, for each $F\in\FCK$,
\begin{equation*}
(\nabla^\M F)(\eta,x)=\bigg(
\frac1{s_x}(\nabla_x^XF)(\eta)
,\,
\frac1{s_x}(\nabla_x^{\Rp}F)(\eta)
\bigg),\quad\eta\in\K(X),\ x\in\tau(\eta)
. \end{equation*}
This leads us to the Dirichlet form
$$ \mathcal E_\lambda^{\mathbb M}(F,G):=\frac12\int_{\K(X)}\langle
\nabla^{\mathbb M}F,\nabla^{\mathbb M}G\rangle_{T(\M)}\,d\mu_\lambda,\quad F,G\in\FCK.$$
Then one can prove a counterpart of Theorem~\ref{vgfyt7r}.  The generator of this Dirichlet form acts as follows: for each $F\in\FCK$,
$$(L_\lambda^\M F)(\eta)=\int_X\bigg[
\frac1{2s_x}(\Delta_x^XF)(\eta)+
(\mathscr L_x^{\Rp}F)(\eta)\bigg]d\tilde\eta(x),$$
where $(\mathscr L_x^{\Rp}F)(\eta)$ is defined analogously to $\Delta_x^{\Rp}$ (see formulas \eqref{uiyt6785o}, \eqref{fctyr567}) by using the operator
$$ (\mathscr L^{\Rp}f)(s)=\frac12\left(sf''(s)+s\,\frac{l'(s)}{l(s)}\,f'(s)\right).$$
Furthermore, by using the theory of Dirichlet forms, it can be shown that,
 under the assumption that the dimension of $X$ is $\ge2$, there exists a conservative diffusion process on $\K(X)$ which is properly associated with  $(\mathcal E^{\M}_\lambda,D(\mathcal E^\M_\lambda))$.

 One could expect that this Markov process has the form $\sum_{i=1}^\infty s_i(t)\delta_{x_i(t)}$, $t\ge0$, in which the pairs $\big((s_i(t),x_i(t)))\big)_{i=1}^\infty$ are independent, and the generator of each Markov process $(s_i(t),x_i(t))$ in $\hat X$ is given by
$$(\mathscr L^{\hat X}g)(s,x)=(\mathscr L_s^{\Rp}g)(s,x)+\frac1{2s}(\Delta_x^X g)(s,x). $$
However, let us consider the special case of the gamma measure, $l(s)=e^{-s}$. Then
$$ (\mathscr L^{\Rp}f)(s)=\frac s2\big(f''(s)- f'(s)\big).$$
Using e.g.\ \cite[Chap.~XI]{RY},
we  conclude that $\mathscr L^{\Rp}$ is the generator of the Markov process $\mathscr Z(t)$ on $[0,\infty)$ given by
$$\mathscr Z(t)=e^{-t}Q((e^{t}-1)/2),$$
where $Q(t)$ is the square of the $0$-dimensional Bessel process.
Note that, for each  starting point $s>0$, the process $\mathscr Z(t)$ is at 0 with probability $\exp(-s/(1-e^{-t/2}))$, and once $\mathscr Z(t)$ reaches zero it stays there forever. Thus, $\mathscr Z(t)$ is not a conservative process on $\Rp$. Hence, it is natural to suggest that the Markov process on $\hat X$ with generator $\mathscr L^{\hat X}$ is also non-conservative. In this case, the explicit construction of a Markov process on $\hat X$ with generator $L_\lambda^\M$ is not clear to the authors even at a heuristic level, compare with \cite{Surgailis2}.

\subsection*{Acknowledgments}
We are grateful to the referees for their careful reading of the manuscript and helpful suggestions.

The authors   acknowledge the financial support of the SFB~701 ``Spectral structures and topological methods in mathematics'' (Bielefeld University) and the Research Group
``Stochastic Dynamics: Mathematical Theory and Applications'' (Center for Interdisciplinary Research, Bielefeld University). AV was partially supported by grants RFBR 14-01-00373 and OFI-M 13-01-12422.

\end{document}